\documentclass[a4paper, 12pt]{amsart}

\usepackage[english]{babel}
\usepackage[latin1]{inputenc}
\usepackage{amsfonts}
\usepackage{amssymb}
\usepackage[leqno]{amsmath}
\usepackage{amsthm}
\usepackage{amscd}
\usepackage[all]{xypic}
\usepackage{enumerate}
\usepackage[mathscr]{eucal}
\usepackage{graphicx}

\newtheorem{theorem}{Theorem}[section]
\newtheorem{proposition}[theorem]{Proposition}
\newtheorem{definition}[theorem]{Definition}

\newtheorem{lemma}[theorem]{Lemma}

\theoremstyle{definition}

\def\SS{\mathbf{S}}
\def\SO{\mathbf{SO}}
\def\SL{\mathbf{SL}}
\def\N{\mathbf{N}}
\def\Z{\mathbf{Z}}
\def\R{\mathbf{R}}
\def\Q{\mathbf{Q}}
\def\C{\mathbf{C}}

\def\GL{\mathbf{GL}}

\renewcommand{\qed}{\quad\hskip0pt\null\hfill$\square$\par}

\title{Veech groups, irrational billiards and stable abelian differentials}
\author[Ferr\'an Valdez]{Ferr\'an Valdez}
\address{Max Planck Institut f\"{u}r Mathematik
Vivatsgasse 7.
53111, Bonn, Germany.}
\email{ferran@mpim-bonn.mpg.de}

\begin{document}

\begin{abstract}
We describe Veech groups  of flat surfaces arising from irrational angled polygonal billiards or
irreducible stable abelian differentials. For irrational polygonal billiards, we prove that these groups are 
non-discrete subgroups of $\rm SO(2,\mathbf{R})$ and we calculate their  rank.  
\end{abstract}

\maketitle

\section{Introduction}
The Veech group of a flat surface is the group of derivatives of orientation-preserving affine homeomorphisms. If the surface is compact, Veech groups are discrete subgroups of $\mathbf{SL}(2,\R)$ that can be related to the geodesic flow on the surface \cite{V}. Our main goal is to describe Veech groups arising from non-compact flat surfaces associated to billiards on an irrational angled polygon. Nevertheless, in this article we will not discuss dynamical aspects of geodesics.  More precisely, 
\begin{theorem}
	\label{theo}
Let $P\subset\R^2$ be a simply connected polygon with interior angles $\{\lambda_j\pi\}_{j=1}^N$, $S(P)$ the flat surface obtained from $P$ via the Katok-Zemljakov construction and $G(S)$ the Veech group of $S(P)$. Suppose there exists $\lambda_j\in\R\setminus\Q$ for some $j=1,\ldots,n$. Then, $G(S)<\SO(2,\R)$ and the group generated by the rotations
\begin{equation}
	\label{RS}
R(S)=<
\begin{pmatrix}
  \cos(2\lambda_j\pi) & -\sin(2\lambda_j\pi) \\
  \sin(2\lambda_j\pi) & \cos(2\lambda_j\pi)
\end{pmatrix}\mid
j=1,\ldots,N>
\end{equation}
has maximal rank in $G(S)$.
\end{theorem}
\indent The surface $S(P)$ has infinite genus and only one end. A topological surface satisfying these two conditions is called a \emph{Loch Ness monster} \cite{Va2}.\\ 
\indent After a suggestion from M. M\"{o}ller, we consider Veech groups arising from stable abelian differentials at $\partial\Omega\overline{M_g}$, the boundary of the Deligne-Mumford compactification of the Hodge bundle $\Omega M_g$ (see \S 5, \cite{B} for a definition). On this boundary,  the notion of Veech group makes sense only for stable abelian differentials on an irreducible stable curve, called irreducible. In this direction, we prove the following:
\begin{proposition}
	\label{prop}
	Let $(X,\omega)\in\partial\rm\Omega\overline{\mathcal{M}_g}$ be an irreducible stable Abelian differential of genus $ g$. Suppose that there exists at least one node in $X$ where the 1--form $\omega$ has a pole. Let $ \{r_j,-r_j\}_{j=1}^k$ be the set of all residues of $\omega$ and define
\begin{equation}
	\label{N}
	\
	N:=<\{
\begin{pmatrix}
 1 & s \\
  0 & t
\end{pmatrix}\mid t\in\R^+, s\in\mathbf{R}\},-Id>.
\end{equation}
Let $G(X)=G(X,\omega)$ be the Veech group of $(X,\omega)$. Then,
\begin{enumerate}
\item If there exist $ i\neq j$ such that $ r_i/r_j\notin\mathbf{R}$, then $G(X)$ is finite.\\
\item If all residues of $\omega$, as vectors in $\mathbf{C}\simeq\mathbf{R}^2$ are parallel, then $G(X)<N$ is either conjugated to  a discrete subgroup or the equal to $N$.
\end{enumerate}
\end{proposition}
\indent Recently, Hubert and Schmith\"{u}sen \cite{HSc} have shown the existence of a countably family of \emph{infinite area}  origamis whose Veech groups are infinitely generated subgroups of $\SL(2,\Z)$. These origamis arise as $\Z$-covers of (finite area) genus 2 origamis. Motivated by this work, in the last section of this article we construct, for each $n\in \N$, an uncountable family of flat surfaces  $\mathcal{S}_n=\{ S_i\}_{i\in I}$ such that each $S_i$ is homeomorphic to the Loch Ness monster and the Veech group $G(S_i)<\SO(2,\R)$ is infinitely generated.\\
\\
\indent This article is organized as follows. We introduce the notion of \emph{tame} flat surface and extend the definition of some classical geometric invariants (saddle connections, Veech groups) to the non-compact realm in Section 2 . Loosely speaking, \emph{tame} flat surfaces present a discrete set of singularities, which are either of finite or infinite angle. We briefly recall the Katok-Zemljakov construction, the notion of stable Abelian differential and define Veech groups for irreducible nodal flat surfaces. Section 3 deals with the proof of Theorem \ref{theo} and Section 4 with the proof of Proposition \ref{prop}. Finally, Section 5 presents the construction of the family of flat surfaces $\mathcal{S}_n$ mentioned above.\\
\\
\textbf{Acknowledgments}. This article was written during a stay of the author at the Max Planck Institut f\"{u}r Mathematik in Bonn. The author wishes to express his gratitude to the administration and staff of the MPI for the wonderful working facilities and the atmosphere. The author acknowledges support from the Sonderforschungsbereich/Transregio 45 and the ANR Symplexe. The author thanks M. M\"{o}ller and M. Bainbridge for valuable discussions.

\section{Preliminaries} 
	\label{preliminaries}
\textbf{Non-compact flat surfaces}. Let $(S,\omega)$ be a pair formed by a connected Riemann surface $S$ and a holomorphic 1--form $\omega$ on $S$ which is not identically zero. Denote by $Z(\omega)\subset S$ the zero locus of the form $\omega$. Local integration of this form endows $S\setminus Z(\omega)$ with an atlas whose transition functions are translations of $\C$. The pullback of the standard translation invariant flat metric on the complex plane defines a flat metric  $d$ on $S\setminus Z(\omega)$. Let $\widehat{S}$ be the metric completion of $S$. Each point in $Z(\omega)$ has a neighborhood isometric to the neighborhood of $0\in \C$ with the metric induced by the 1--form $z^kdz$ for some $k>1$ (which is a cyclic finite branched covering of $\C$). The points in $Z(\omega)$ are called $\emph{finite angle singularities}$.
 Note that there is a natural embedding of $S$ into $\widehat{S}$.
\begin{definition}
A point $p\in \widehat{S}$ is called an \emph{infinite angle singularity}, if there exists a radius $\epsilon>0$ such that the punctured neighborhood:
\begin{equation}
	\label{infiniteangle}
	\{z\in\widehat{S}\mid 0<d(z,p)<\epsilon\}
\end{equation}
is isometric to the infinite cyclic covering of $\epsilon\mathbf{D}^*=\{w\in\C^*\mid0<\mid w\mid <\epsilon\}$. We denote by $Y_{\infty}(\omega)$ the set of infinite angle 
singularities of $\widehat{S}$.

\end{definition}
\begin{definition}
	\label{flatsurface}
	The pair $(S,\omega)$ is called a \emph{tame} flat surface if $\hat{S}\setminus S=Y_{\infty}(\omega)$.
	\end{definition}
	\indent One can easily check that flat surfaces arising from irrational polygons or stable abelian differentials are tame.
\begin{definition}
	\label{singgeo}
	A \emph{singular geodesic} of $S=(S,\omega)$ is an open geodesic segment in the flat metric $d$ whose image under the natural embedding $S\hookrightarrow\widehat{S}$ issues from a singularity of $\widehat{S}$, contains no point of $Y(\omega)$ in its interior and is not properly contained in some other geodesic segment. A \emph{saddle connection} is a finite length singular geodesic.  
\end{definition}
\indent To each saddle connection we can associate a \emph{holonomy vector}: we 'develop' the saddle connection in the plane by using local coordinates of the flat structure. The difference vector defined by the planar line segment is the holonomy vector. Two saddle connections are \emph{parallel}, if their corresponding holonomy vectors are linearly dependent. \\
\indent Let $\mathrm{Aff}_+(S)$ be the group of affine orientation preserving homeomorphisms of the flat surface $S$ (by definition $S$ comes with a distinguished 1--form $\omega$). Consider  the map \begin{equation}
	\label{eseq}
	 \mathrm{Aff}_+(S)\overset{D}\longrightarrow \mathbf{GL}_+(2,\R)
\end{equation}
 that associates to every $\phi\in \mathrm{Aff}_+(S)$ its (constant) Jacobian derivative $D\phi$. 
 \begin{definition}
 	\label{vgroup}
	Let $S$ be a flat surface. We call $G(S)=D(\mathrm{Aff}_+(S))$ the \emph{Veech group} of $S$.
 \end{definition} 
 \textbf{The Katok-Zemljakov construction}. In the following paragraph we recall briefly the definition of this construction. For details see \cite{Va2} and references within.\\
 \indent  Let $P_0$ denote the polygon $P$ deprived of its vertices. The identification of two disjoint copies of $P_0$ along "common sides" defines a Euclidean structure on the $N$-punctured sphere. We denote it by $\SS^2(P)$. This punctured surface is naturally covered by $S(P_0)$, the \emph{minimal translation surface} corresponding to $P$. We denote the projection of this covering by $\pi:S(P_0)\longrightarrow\SS^2(P)$. Call a vertex of $P$ \emph{rational}, if the corresponding interior angle is commensurable with $\pi$.
 When the set of rational vertices of $P$ is not empty, the translation surface $S(P_0)$ can be locally compactified by adding points "above" rational vertices of $P$. The result of this local compactification is a flat surface with finite angle singularities that we denote by $S(P)$.  If the set of rational vertices of $P$ is empty, we set $S(P)=S(P_0)$. In both cases, $S(P)$ is called the flat surface obtained from the polygon $P$ via the \emph{Katok-Zemljakov construction}. Remark that, in the case of rational polygons, some authors give a different definition (see \cite{MT} or \cite{GuT}).\\
 \\
\textbf{Stable Abelian differentials}. We recall briefly the notion of stable Abelian differential, following Bainbridge \cite{B}.\\
\indent A \emph{nodal Riemann surface} $ X$ is a finite type Riemann surface, \emph{i.e.} with finitely generated fundamental group, that has finitely many cusps which have been identified pairwise to form nodes. A connected component of a nodal Riemann surface $ X$ with its nodes removed is called a \emph{part} of $ X$, and the closure of a part of $ X$ is an \emph{irreducible component}. The genus of a nodal Riemann surface is the topological genus of the non-singular Riemann surface obtained by replacing each node in $ X$ with an annulus. A \emph{stable Riemann surface} is a connected nodal Riemann surface for which each part has negative Euler characteristic. A \emph{stable Abeliann differential} $\omega$ on a stable Riemann surface $ X$ is a holomophic 1--form on $ X$ minus its nodes such that its restriction to each part of $ X$ has at worst simple poles at the cusps and at two cusps which have been identified to form a node, the differential has opposite residues, if any. Nodes at which $\omega$ presents a pole are called \emph{polar nodes}.\\
\\
\textbf{Veech groups on stable Abelian differentials}. Let $(X,\omega)$ be a stable Abelian differential. We denote by ${\rm Aff_+}(X,\omega)$ the group of affine orientation preserving homeomorphisms of $X$. The Jacobian derivative $D\phi$ of an affine homeomorphism $\phi$ is constant on each irreducible component of $(X,\omega)$. In general, there is no canonical derivation morphism from the affine group of a stable Abelian differential onto $\GL_+(2,\R)$. Consider, for example, the genus 2 stable Abelian differential given by the following figure:
\begin{center}
  \label{nends}
\includegraphics[scale=0.3]{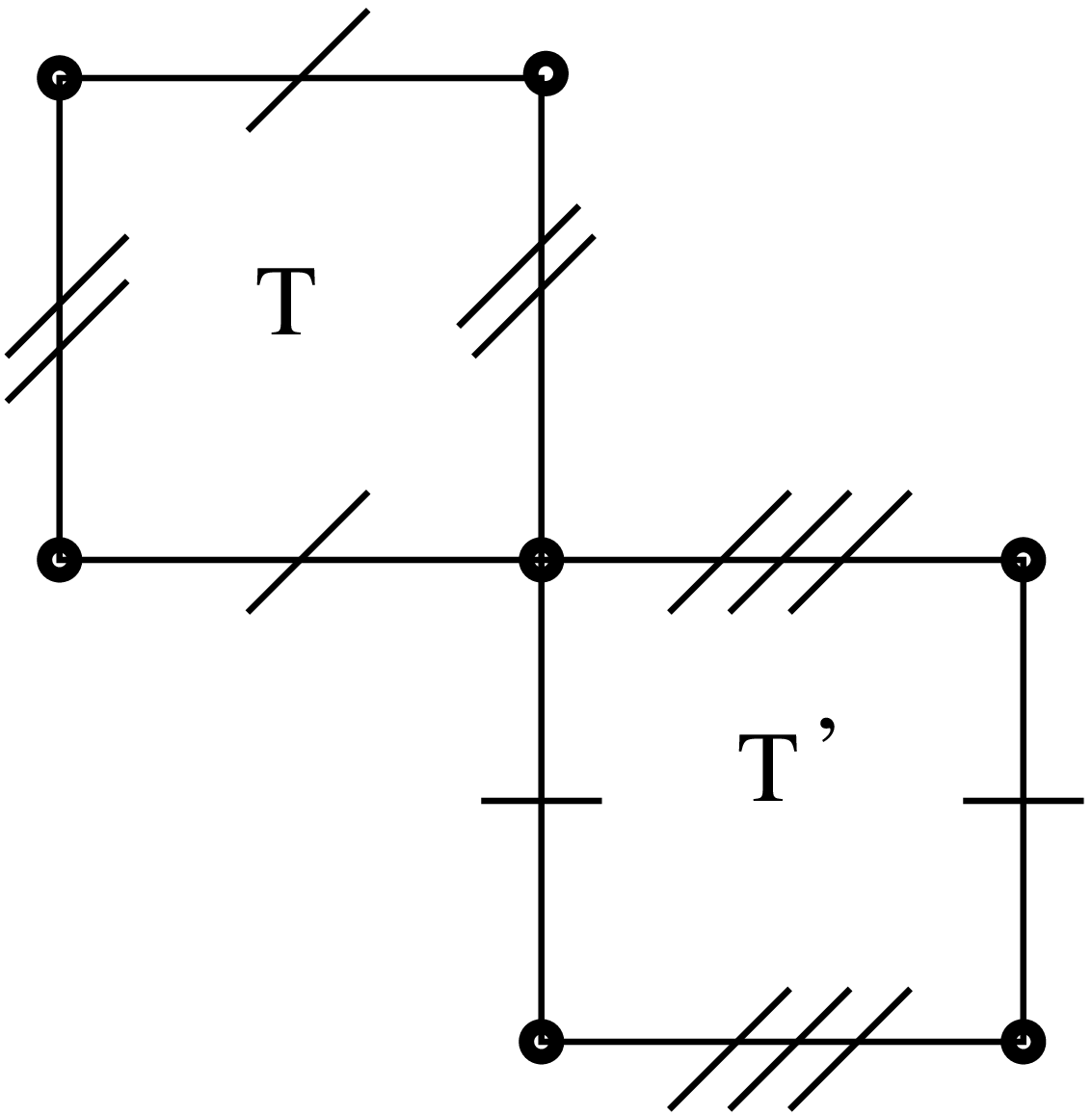}\\
Figure 1.
\end{center}
We avoid this situation by restricting ourselves to irreducible Riemann surfaces. 
\begin{definition}
Let $X=(X,\omega)$ be an \emph{irreducible} stable Abelian differential. We call $G(X)=D({\rm Aff_+}(X,\omega))$ the \emph{Veech group} of $X$.
\end{definition}
\indent Abelian differentials close to a stable Abelian differential $\rm (X,\omega)$ with a polar node develop very long cylinders which are pinched off to form a node in the limit, (see \S 5.3 \cite{B}). In the following figure we depict a genus two stable abelian differential with two nodes ( with residues $\rm \pm 1$ and $\rm \pm (1+i)$) and two double zeroes:
\begin{center}
\includegraphics[scale=0.28]{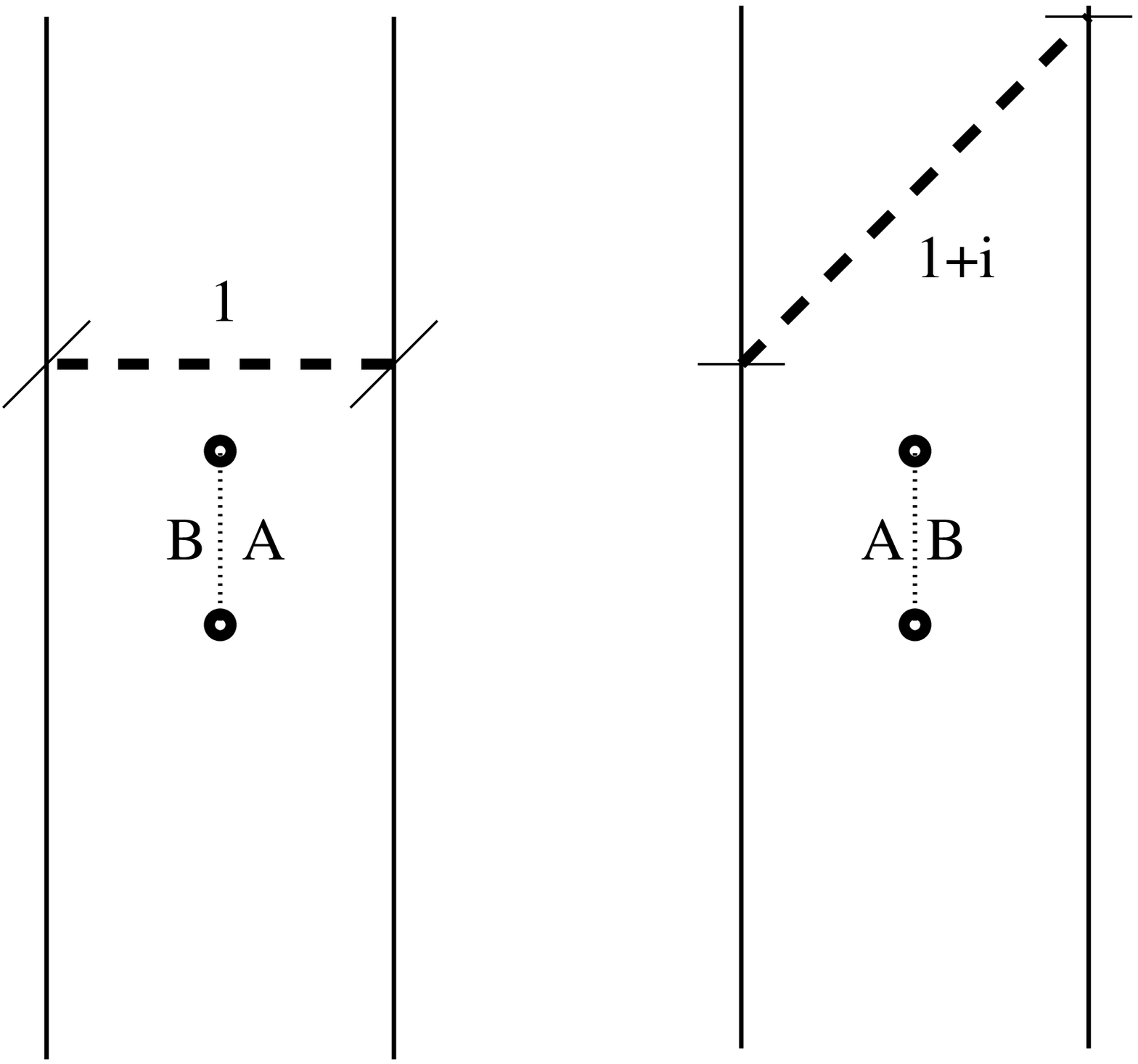}\\
Figure 2.
\end{center}
\indent When considering the flat metric, every stable Abelian differential deprived of its polar nodes is a complete metric space. We call \emph{singular geodesic} in the context of stable Abelian differentials, every geodesic segment that issues from a zero or a non-polar node of $\omega$, contains no such zero or non-polar node on its interior and is not properly contained in some other geodesic segment. As before, finite length singular geodesics will be called \emph{saddle connections}.\\
\\
\emph{Decomposition of stable Abelian differentials with polar nodes}. Suppose that $(X,\omega)$ has polar nodes with residues $r_1,\ldots,r_k$.  Every $ r_j$ defines a direction $\theta(r_j)\in\mathbf{R}/\mathbf{Z}$ for which $ (X,\omega)$ presents a set of disjoint infinite area cylinders $C_{1,j},\ldots, C_{n(j),j}$ foliated by closed geodesics parallel to $\theta(r_j)$ and whose length is $\mid r_j\mid$. Denote by $C_j$ the closure in $(X,\omega)$ of $\cup_{i=1}^{n(j)}C_{i,j}$ and $C=\cup_{j=1}^k C_j$. We define
\begin{equation}
	\label{Xprime}
 X':=X\setminus C	
\end{equation}
\indent The Veech group of $(X,\omega)$ acts linearly on the set of residues of $\omega$ and leaves the decomposition $X=X'\sqcup C$ invariant.

\section{Proof of Theorem \ref{theo}}
\label{pmtheo}
First, we prove that the matrix group $R(S)$ defined in (\ref{RS}) is a subgroup of $G(S)$. Then, we prove that $G(S)<\SO(2,\R)$ and, finally, that ${\rm Rank}(G(S))={\rm Rank}(R(S))$.\\
\\
\textbf{(i)} The locally Euclidean structure on the $N$-punctured sphere $\mathbf{S}^2(P)$ gives rise to the holonomy representation:
\begin{equation}
  \label{hol}
  \rm
 hol:\pi_1(\mathbf{S}^2(P))\longrightarrow Isom_+(\mathbf{R}^2)
\end{equation}
\indent  Let $B_j $ be a simple loop in $ \mathbf{S}^2(P)$ around the missing vertex of $ P$ whose interior angle is $\lambda_j\pi$, $\rm j=1,\ldots, N$. Suppose that $ B_j\cap B_i=*$, for $\rm i\neq j$. Then, $ \{B_j\}_{j=1}^N$ generates $\pi_1(\mathbf{S}^2(P),*)$. Given an isometry $\varphi\in \rm Isom_+(\mathbf{R}^2)$, we denote its derivative by $ D\circ\varphi$. A direct calculation in local coordinates shows that $\rm hol(B_j)$ is affine and that $M_j= D\circ \rm hol(B_j)$ is given by:
  \begin{equation}
  \label{rots}
  M_j=
\begin{pmatrix}
  \cos(2\lambda_j\pi) & -\sin(2\lambda_j\pi) \\
  \sin(2\lambda_j\pi) & \cos(2\lambda_j\pi)
\end{pmatrix}
\hspace{1cm} j=1,\ldots,N.
 \end{equation}
 \indent Since $G(S(P_0))=G(S(P))$, we conclude that $R(S)$ is a subgroup of $G(S)$.\\
\textbf{(ii)} We claim that length of every saddle connection in $S(P)$ is bounded below by some constant $c=c(P)>0$. Indeed, consider the folding map $f:\SS^2(P)\longrightarrow P$ which is 2-1
except along the boundary of $P$. The projection $f\circ \pi:S(P_0)\longrightarrow P$ maps every saddle connection $\gamma\subset S(P_0)$ onto a \emph{generalized diagonal} of the billiard game on $P$ (see \cite{K} for a precise definition). The length of $\gamma$ is bounded below by the length of the generalized diagonal $f\circ\pi(\gamma)$. The length of any generalized diagonal of the billiard table  $P$ is bounded below by some positive constant $c$ depending only on $P$. This proves our claim. The constant $c$ is realized by a generalized diagonal. Therefore, we can choose a  holonomy vector $v$ is of minimal length. Given that $R(S)<G(S)$, the $G(S)$-orbit of $v$ is dense in the circle of radius $|v|$ centered at the origin. This forces the Veech group $G(S)$ to lie in $\SO(2,\R)$.\\
\\
\textbf{(iii)} Suppose that there exist an affine homeomorphism $\varphi\in {\rm Aff}_+(S)$ such that $D\varphi$ is an infinite order element of $\SO(2,\R)/R(S)$. Let $\gamma_0$ be a fixed saddle connection. Then $\{f\circ\pi\circ\varphi^k(\gamma_0)\}_{k\in Z}$ is an infinite set of generalized diagonals of bounded length. But this is a contradiction, for the set of generalized diagonals of bounded length on a polygonal billiard is always finite \cite{K}.

\section{Proof of Proposition \ref{prop}}
	\label{boundarycase}
The Veech group of the irrational stable Abelian differential $(X,\omega)$ acts linearly on the (finite) set of residues of $\omega$. Therefore, if not all residues are parallel, $G(X)$ must be finite. \\
\indent Suppose now that all residues are parallel to the horizontal direction. Then $G(X)<N$. If every holonomy vector of $(X,\omega)$ is horizontal we claim that $G(X)=N$. Indeed, in this situation $X'$ defined in (\ref{Xprime}) is empty and the horizontal geodesic flow decomposes $X$ into finitely many cylinders with horizontal boundaries. This allows to define, for every $g\in N$, an orientation preserving affine homeomorphism of $X$ whose differential is exactly $g$. On the other hand, if at least one holonomy vector fails to be horizontal, then $G(X)<N$ is discrete, for the set of holonomy vectors of any stable Abelian differential is discrete.\\ \qed
\textbf{Remark}. Veech groups of irreducible stable Abelian differentials in $\partial\Omega\overline{\mathcal{M}_g}$ without polar nodes are as "complicated" as Veech groups of flat surfaces in $\Omega\mathcal{M}_g$ with marked points. More precisely, a nodal Riemann surface $X$ has a \emph{normalization} $ f: S(X)\longrightarrow X$ defined by separating the two branches passing through each node of $ X$. For every node $ p$, denote $\{p_+,p_-\}:=f^{-1}(p)$. Then, if the stable Abelian differential $(X,\omega)$ has no polar nodes, we have the equality:
\begin{equation}
	\label{notpolar}
	{\rm Aff}_+(X,\omega)=\{\phi\in {\rm Aff}_+(S(X),\omega)\hspace{1mm}|\hspace{1mm} \phi(p_+)=\phi(p_-),\hspace{1mm}\forall\hspace{1mm}p\hspace{1mm}\text{node of $X$}\}.
\end{equation}

\section{Infinitely generated Veech groups in $\SO(2,\R)$}
	\label{finalremarks}
\indent Fix $n\in\N$.  Consider an unbounded sequence of real numbers 
\begin{equation}
	\label{lasec}
	\rm x_0=0<x_1<x_2<\ldots<x_j<... 
\end{equation}	
such that $\rm x_{j+1}-x_j>1$ for all $j$. The segments of straight line joining the point $\rm (x_j,x_j^{2n})$ to\linebreak  $\rm (x_{j+1},x_{j+1}^{2n})$ and  $\rm (-x_j,x_j^{2n})$ to $\rm (-x_{j+1},x_{j+1}^{2n})$, $j\geq 0$, define a polygonal line $ \partial P$ in $\mathbf{C}$. Let $ int(P)$ be the connected component of  $ \mathbf{C}\setminus \partial P$ intersecting the positive imaginary axis $\rm Im(z)>0$. We define $ P=\partial P\cup int(P)$. We call $P$ the \emph{unbouded polygon} defined by the \linebreak sequence (\ref{lasec}). Remark that $P$ is symmetric with respect to the imaginary axis. For each $\rm j\geq 0$, let $\rm \lambda_j\pi$ be the interior angle of $ P$ at the vertex $\rm (x_j,x_j^{2n})$.\\
\begin{definition}
	\label{infires}
We say that a sequence of real numbers $\rm \{\mu_j\}_{j\geq 0}$ is \emph{free of resonances} if and only if for every finite subset $\rm \{\mu_{j_1},\ldots \mu_{j_N}\}$ the kernel of the group morphism $\Z^N\longrightarrow \C$ defined by
$$
\rm
(n_1,\ldots,n_N)\longrightarrow exp(2\pi i(\sum_{k=1}^Nn_k\mu_{j_k}))
$$
is trivial.\\
\end{definition}
\indent There are uncountable many choices $\{x_{j\geq 0}^i\}$, $i\in I$, for (\ref{lasec}), such that the sequence $\{\lambda_j^i\}_{j\geq 0}$ defining the interior angles of $P=P_i$ is free of resonances. For each $i\in I$, denote by $S^2(P_i)$ the identification of two vertexless copies of $P_i$ along "common sides". The Katok-Zemljakov construction described in Section \ref{preliminaries} can be applied to the unbounded polygon $P_i$. The result is a flat covering $S_i\longrightarrow \mathbf{S}^2(P_i)$.
\begin{lemma}
	\label{infelliptic}
	The flat surface $S_i$ is homeomorphic to the Loch Ness monster. The Veech group $G(S_i)<\SO(2,\R)$ is infinitely generated and contains the infinite rank group generated by the matrices
\begin{equation}
  \label{rotros}
\begin{pmatrix}
  \cos(2\lambda_j^i\pi) & -\sin(2\lambda_j^i\pi) \\
  \sin(2\lambda_j^i\pi) & \cos(2\lambda_j^i\pi)
\end{pmatrix}
\rm
\hspace{1cm} j\geq 0,
 \end{equation}		
\end{lemma} 
\begin{proof} Every flat surface obtained via the Katok-Zemljakov construction from a (bounded) polygon whose angles are all irrational multiples of $\pi$, is homeomorphic to a Loch-Ness monster. This is proved in Theorem 1 ( Case (A) and absence of resonances) in \cite{Va2}. The same conclusion can be drawn for the unbounded polygons $P_i$ after   replacing in the proof of Theorem 1 [\emph{Ibid.}]  \emph{polygon P} and surface $X(P)$ by \emph{unbounded polygon} $P_i$ and $S_i$, respectively.\\
\indent In \S\ref{pmtheo}, sections (i) and (ii) made use of the boundness of $P$ to assure that 
the length of every saddle connection in $S(P)$ was bounded below by a constant depending only on $P$. For unbounded polygons, this is ensured by the condition  $\rm x_{j+1}-x_j>1$, for all $j$, on the sequence (\ref{lasec}). It follows that, for every $i\in I$,  $G(S_i)<\SO(2,\R)$ and that this Veech group contains the group generated by the matrices (\ref{rotros}). This matrix group is infinitely generated, for the sequence $\{\lambda_j^i\}_{j\geq 0}$ is free of resonances, for every $i\in I$. \\
\end{proof}



\begin{thebibliography}{10}
\bibitem[1]{B}M. Bainbridge. \emph{Euler characteristics of Teichmüller curves in genus two}. Geom. Topol. 11 (2007), 1887--2073. 
\bibitem[2]{GuT}E. Gutkin and S. Troubetzkoy \emph{Directional flows and strong recurrence for polygonal billiards}.  International Conference on Dynamical Systems (Montevideo, 1995),  21--45, Pitman Res. Notes Math. Ser., 362, Longman, Harlow, 1996.
\emph{Geom. Dedicata} 125 (2007), 39--46. 
\bibitem[3]{HSc} P. Hubert and G. Schmith\"{u}sen. \emph{Infinite translation surfaces with infinitely generated Veech groups}. Preprint. http://www.cmi.univ-mrs.fr/~hubert/articles/hub-schmithuesen.pdf 
\bibitem[4]{K} A. B. Katok,
\emph{The growth rate for the number of singular and periodic orbits for a polygonal billiard}. Comm. Math. Phys. 111 (1987), no. 1, 151--160. 
\bibitem[5]{MT}H. Masur and S. Tabachnikov. \emph{Rational billiards and flat structures}. Handbook of dynamical systems. Vol. 1A, 1015--1089, North Holland. Amsterdam 2002.
\bibitem[6]{Va2}J.F. Valdez. \emph{Infinite genus surfaces and irrational polygonal billiards}. To appear in Geom. Dedicata.
\bibitem[7]{V} W.A. Veech. \emph{Teichm\"{u}ller curves in the moduli space, Eisenstein series and applications to triangular billiards}. Inventiones mathematicae, 97, 1989, 553--583.

\end{thebibliography}
\end{document}